\documentclass{amsart}
\usepackage{latexsym,amssymb,amsmath,amsthm,amscd,graphicx}

\makeatletter
\@namedef{subjclassname@2010}{%
  \textup{2010} Mathematics Subject Classification}
\makeatother

\numberwithin{equation}{section}
\newtheorem{theorem}{Theorem}[section]
\newtheorem{lemma}[theorem]{Lemma}

\theoremstyle{definition}

\newtheorem{definition}[theorem]{Definition}

\theoremstyle{remark}

\newcommand{\cA}{{\mathcal A}}

\newcommand{\cD}{{\mathcal D}}
\newcommand{\cE}{{\mathcal E}}
\newcommand{\cF}{{\mathcal F}}
\newcommand{\cG}{{\mathcal G}}

\newcommand{\cL}{{\mathcal L}}

\newcommand{\cP}{{\mathcal P}}

\newcommand{\R}{{\mathbb R}}

\newcommand{\Z}{{\mathbb Z}}

\def\al{\alpha}

\def\kp{\kappa}
\def\lm{\lambda}
\def\sg{\sigma}
\def\Gm{\Gamma}

\def\6{\partial}

\def\l{\left}
\def\r{\right}
\def\ds{\displaystyle}

\begin{document}
\title
[A characterization for the boundedness of positive operators]
{A characterization for the boundedness of positive operators in a filtered measure space}
\author[H.~Tanaka]{Hitoshi Tanaka}
\address{Graduate School of Mathematical Sciences, The University of Tokyo, Tokyo, 153-8914, Japan}
\email{htanaka@ms.u-tokyo.ac.jp}
\thanks{
The first author is supported by 
the Global COE program at Graduate School of Mathematical Sciences, the University of Tokyo, 
Grant-in-Aid for Scientific Research (C) (No.~23540187), 
the Japan Society for the Promotion of Science, 
and was supported by F\=ujyukai foundation.
}
\author[Y.~Terasawa]{Yutaka Terasawa}
\address{Graduate School of Mathematical Sciences, The University of Tokyo, Tokyo, 153-8914, Japan}
\email{yutaka@ms.u-tokyo.ac.jp}
\thanks{
The second author is a Research Fellow of the Japan Society for the Promotion of Science.
}
\date{}

\begin{abstract}
In terms of Sawyer type checking condition, 
a complete characterization is established 
for which the positive operator in a filtered measure space is bounded from $L^p(d\mu)$ to $L^q(d\mu)$ 
with $1<p\le q<\infty$. 
\end{abstract}

\maketitle

\section{Introduction}\label{sec1}
The purpose of this paper is to establish a complete characterization 
for which the positive operator in a filtered measure space as we introduced in \cite{TaTe}
becomes bounded from $L^p(d\mu)$ to $L^q(d\mu)$ 
with $1<p\le q<\infty.$
In his elegant paper \cite{Tr} 
Sergei Treil gives a simple proof of 
the following Sawyer type characterization 
of the two weight estimate for positive dyadic operators: 

\begin{quote}{\it 
Let 
$\cD$ be a dyadic lattice in $\R^n$, 
$\al_{I}:\,\cD\to[0,\infty)$ be a map 
and $\mu$ and $\nu$ be Radon measures in $\R^n$. 
Let $1<p<\infty$ and $1/p+1/p'=1$.
Then 
$$
\l\|
\sum_{I\in\cD}
\al_{I}\int_{I}f\,d\mu\,1_{I}
\r\|_{L^p(d\nu)}
\le
C_1\|f\|_{L^p(d\mu)}
$$
holds if and only if, 
for all $I_0\in\cD$, 
$$
\l\{\begin{array}{l}
\int_{I_0}\l(\sum_{I\in\cD:\,I\subset I_0}\al_{I}\mu(I)1_{I}\r)^p\,d\nu
\le
C_2^p\mu(I_0),
\\
\int_{I_0}\l(\sum_{I\in\cD:\,I\subset I_0}\al_{I}\nu(I)1_{I}\r)^{p'}\,d\mu
\le
C_2^{p'}\nu(I_0).
\end{array}\r.
$$
Moreover,
the least possible $C_1$ and $C_2$ are equivalent.
}\end{quote}

This theorem was first proved for $p=2$ in \cite{NaTrVo} by the Bellman function method. 
Later in \cite{LaSaUr} 
this theorem was proved in full generality; 
the case $L^p(d\mu)\to L^q(d\nu)$, $1<p\le q<\infty$, 
was treated there. However, 
the construction was complicated and depends very much on the dyadic structure.
The proof due to Treil is very simple and, 
as is mentioned in the end of his paper, 
all the proofs work well in a more general martingale situation {\it with} a lattice structure. 
In this paper, 
we shall extend this theorem to a more general martingale situation {\it without} a lattice structure
for the case $\mu=\nu$ (Theorem \ref{thm1.1}). 
Our main idea in the proof is substitution 
of \lq\lq the stopping moments" in \cite{Tr}, which is also called "principal cubes" in other literatures,
for \lq\lq the principal sets", 
which was first used in \cite{TaTe}. 
\lq\lq The stopping moments" 
and 
\lq\lq the principal sets" 
are constructed from a data function $f$. 
If we assume a lattice structure, 
we can assure 
the commonality of constructions made from different data functions $f$ and $g$. 
Not assuming a lattice structure, 
we can expect no commonality, and, 
this is a reason why 
we can not extend the above theorem to a more general martingale situation for the two weight case. 
The checking condition in the above theorem is called 
\lq\lq Sawyer type checking condition", 
since this was first introduced by Eric Sawyer 
in \cite{Sa1,Sa2}. 

Let a triplet $(\Omega,\cF,\mu)$ be a measure space. 
Denote by $\cF^0$ the collection of sets in $\cF$ with finite measure. 
The measure space $(\Omega,\cF,\mu)$ is called $\sg$-finite 
if there exist sets $E_i\in\cF^0$ such that
$\bigcup_{i=0}^{\infty}E_i=\Omega$.
In this paper all measure spaces are assumed to be $\sg$-finite. 
Let $\cA\subset\cF^0$ be an arbitrary subset of $\cF^0$. 
An $\cF$-measurable function 
$f:\,\Omega\to\R$ 
is called $\cA$-integrable if 
it is integrable on all sets of $\cA$, 
i.e., 
$$
1_{E}f\in L^1(\cF,\mu)
\text{ for all }
E\in\cA.
$$
Denote the collection of all such functions by
$L_{\cA}^1(\cF,\mu)$.

If $\cG\subset\cF$ is another $\sg$-algebra, 
it is called a sub-$\sg$-algebra of $\cF$. 
A function 
$g\in L_{\cG^0}^1(\cG,\mu)$ 
is called the conditional expectation of 
$f\in L_{\cG^0}^1(\cF,\mu)$ 
with respect to $\cG$ if there holds
$$
\int_{G}f\,d\mu=\int_{G}g\,d\mu
\text{ for all }
G\in\cG^0.
$$
The conditional expectation of $f$ with respect to $\cG$ 
will be denoted by $E[f|\cG]$, 
which exists uniquely in 
$L_{\cG^0}^1(\cG,\mu)$ 
due to $\sg$-finiteness of $(\Omega,\cG,\mu)$.

A family of sub-$\sg$-algebras 
$(\cF_i)_{i\in\Z}$ 
is called a filtration of $\cF$ if 
$\cF_i\subset\cF_j\subset\cF$ 
whenever $i,j\in\Z$ and $i<j$. 
We call a quadruplet 
$(\Omega,\cF,\mu;(\cF_i)_{i\in\Z})$ 
a $\sg$-finite filtered measure space.
We write 
$$
\cL
:=
\bigcap_{i\in\Z}L_{\cF_i^0}^1(\cF,\mu).
$$
Notice that 
$L_{\cF_i^0}^1(\cF,\mu)
\supset
L_{\cF_j^0}^1(\cF,\mu)$ 
whenever $i<j$. 
For a function $f\in\cL$ 
we will denote $E[f|\cF_i]$ by $\cE_if$. 
By the tower rule of conditional expectations, 
a family of functions
$\cE_if\in L_{\cF_i^0}^1(\cF_i,\mu)$ 
becomes a martingale. 
By a weight we mean a nonnegative function 
which belongs to $\cL$ and, 
by a convention, 
we will denote the set of all weights by $\cL^{+}$. 

Let $\al_i$, $i\in\Z$, be a 
nonnegative bounded $\cF_i$-measurable function 
and set $\al=(\al_i)$. 
For a function $f\in\cL$ 
we define a positive operator $T_{\al}f$ by 
$$
T_{\al}f:=\sum_{i\in\Z}\al_i\cE_if.
$$
The following is our main theorem. 

\begin{theorem}\label{thm1.1}
Let $1<p\le q<\infty$. Then 
\begin{equation}\label{1.1}
\|T_{\al}f\|_{L^q(d\mu)}
\le
C_1\|f\|_{L^p(d\mu)}
\end{equation}
holds if and only if 
\begin{equation}\label{1.2}
\sup_{i\in\Z}\sup_{E\in\cF_i^0}
\mu(E)^{\frac1q-\frac1p-\frac1r}
\l(\int_{E}\l(\sum_{j\ge i}\al_j\r)^r\,d\mu\r)^{\frac1r}
\le C_2<\infty,
\end{equation}
where $r:=\max(q,p')$. 
Moreover, 
the least possible $C_1$ and $C_2$ are equivalent.
\end{theorem}
We remark that the checking condition of Theorem \ref{thm1.1} 
looks like the Morrey norm (see, for example, \cite{SaSuTa}). 
For a Morrey spaces in a filtered measure spaces with lattice structure see \cite{NaSa}. 

For a function $f,g\in\cL$ 
we define a positive bilinear operator $T_{\al}(f,g)$ by 
$$
T_{\al}(f,g)
:=
\sum_{i\in\Z}\al_i(\cE_if)(\cE_ig).
$$
Theorem \ref{thm1.1} can be proved easily by the following\footnote{
Notice that 
$$
\max\l\{
\l(\frac{1}{\mu(E)}\int_{E}\l(\sum_{j\ge i}\al_j\r)^q\,d\mu\r)^{\frac1q}
\,,\,
\l(\frac{1}{\mu(E)}\int_{E}\l(\sum_{j\ge i}\al_j\r)^{p'}\,d\mu\r)^{\frac1{p'}}
\r\}
=
\l(\frac{1}{\mu(E)}\int_{E}\l(\sum_{j\ge i}\al_j\r)^r\,d\mu\r)^{\frac1r}.
$$
}.

\begin{theorem}\label{thm1.2}
Let $1<p\le q<\infty$. Then 
\begin{equation}\label{1.3}
\|T_{\al}(f,g)\|_{L^1(d\mu)}
\le C_1
\|f\|_{L^p(d\mu)}\|g\|_{L^{q'}(d\mu)}
\end{equation}
holds if and only if, 
for any $E\in\cF_i^0$, $i\in\Z$, 
\begin{equation}\label{1.4}
\l\{\begin{array}{l}
\l(\int_{E}\l(\sum_{j\ge i}\al_j\r)^q\,d\mu\r)^{\frac1q}
\le C_2\mu(E)^{\frac1p},
\\
\l(\int_{E}\l(\sum_{j\ge i}\al_j\r)^{p'}\,d\mu\r)^{\frac1{p'}}
\le C_2\mu(E)^{\frac1{q'}}.
\end{array}\r.
\end{equation}
Moreover, 
the least possible $C_1$ and $C_2$ are equivalent.
\end{theorem}

The letter $C$ will be used for constants
that may change from one occurrence to another.
Constants with subscripts, such as $C_1$, $C_2$, do not change
in different occurrences.

\section{Proof of Theorem \ref{thm1.2}}
In what follows 
we shall prove Theorem \ref{thm1.2}. 
We first list three basic properties of the conditional expectation 
and the definition of a martingale. 

Let $(\Omega,\cF,\mu)$ be a $\sg$-finite measure space 
and $\cG$ be a sub-$\sg$-algebra of $\cF$. 
Then the following holds.

\begin{description}
\item[{\rm(i)}] 
Let 
$f\in L_{\cG^0}^1(\cF,\mu)$ 
and $g$ be a $\cG$-measurable function. 
Then the two conditions 
$fg\in L_{\cG^0}^1(\cF,\mu)$ 
and 
$gE[f|\cG]\in L_{\cG^0}^1(\cG,\mu)$ 
are equivalent and we have 
$$
E[fg|\cG]=gE[f|\cG];
$$
\item[{\rm(ii)}] 
Let 
$f_1,f_2\in L_{\cG^0}^1(\cF,\mu)$. 
Then the three conditions 
$$
E[f_1|\cG]f_2\in L_{\cG^0}^1(\cG,\mu),
\quad
E[f_1|\cG]E[f_2|\cG]
\in L_{\cG^0}^1(\cG,\mu)
\text{ and }
f_1E[f_2|\cG]
\in L_{\cG^0}^1(\cG,\mu)
$$
are all equivalent and we have 
$$
E[E[f_1|\cG]f_2|\cG]
=
E[f_1|\cG]E[f_2|\cG]
=
E[f_1E[f_2|\cG]|\cG];
$$
\item[\rm{(iii)}]
Let 
$\cG_1\subset\cG_2\subset\cF$ 
be two sub-$\sigma$-algebras of $\cF$ and 
let $f\in L_{\cG_2^0}(\cF,\mu)$. 
Then 
$$
E[f|\cG_1]=E[E[f|\cG_2]|\cG_1].
$$
\end{description}

\noindent
(i) can be proved by an approximation by simple functions. 
The property (ii) means that 
conditional expectation operators are selfadjoint and
can be easily deduced from (i). (iii) can be proved easily
and called the tower rule of conditional expectations.

\begin{definition}\label{def2.1}
Let 
$(\Omega,\cF,\mu;(\cF_i)_{i\in\Z})$ 
be a $\sg$-finite filtered measure space.
Let $(f_i)_{i\in\Z}$ be a sequence of 
$\cF_i$-measurable functions. 
Then the sequence $(f_i)_{i\in\Z}$ is called a 
\lq\lq martingale" if 
$f_i\in L_{\cF_i^0}^1(\cF_i,\mu)$
and 
$f_i=\cE_if_j$ 
whenever $i<j$. 
\end{definition}

From now on we start the proof of Theorem \ref{thm1.2}. 
Without loss of generality we may assume that 
$f$ and $g$ belong class $\cL^{+}$ and 
are bounded and have the supports on the finite measure. 
The \lq\lq only if" part can be easily proved. 
Since we have 
$$
\int_{\Omega}\l(\sum_i\al_i(\cE_if)\r)g\,d\mu
=
\sum_i
\int_{\Omega}\al_i(\cE_if)g\,d\mu
=
\sum_i
\int_{\Omega}\al_i(\cE_if)(\cE_ig)\,d\mu
\le C_1
\|f\|_{L^p(d\mu)}\|g\|_{L^{q'}(d\mu)},
$$
where we have used (ii). 
By a duality argument, 
if we take $f=1_{E}$, 
this implies the condition \eqref{1.4}. 
We shall prove \lq\lq if" part of Theorem \ref{thm1.2}. 

Let $i_0\in\Z$ be arbitrarily taken and be fixed. 
By a standard limiting argument, 
it suffices to prove Theorem \ref{thm1.2} that the inequality
\begin{equation}\label{2.1}
\sum_{i\ge i_0}
\int_{\Omega}
\al_i(\cE_if)(\cE_ig)\,d\mu
\le C
\l\{\|f\|_{L^p(d\mu)}^{p\theta}+\|g\|_{L^{q'}(d\mu)}^{q'\theta}\r\},
\quad\theta:=\frac1p+\frac1{q'},
\end{equation}
holds (the rest follows from the homogeneity). 

We set 
$$
F_i
:=
\{(\cE_ig)^{q'}\le(\cE_if)^p\}
\text{ and }
G_i
:=
\Omega\setminus F_i.
$$
We notice that $F_i,G_i\in\cF_i$.
We shall prove 
\begin{equation}\label{2.2}
\sum_{i\ge i_0}
\int_{\Omega}
1_{F_i}\al_i(\cE_if)(\cE_ig)\,d\mu
\le C
\|f\|_{L^p(d\mu)}^{p\theta}
\end{equation}
and
\begin{equation}\label{2.3}
\sum_{i\ge i_0}
\int_{\Omega}
1_{G_i}\al_i(\cE_if)(\cE_ig)\,d\mu
\le C
\|g\|_{L^{q'}(d\mu)}^{q'\theta}.
\end{equation}

Since the proofs of \eqref{2.2} and \eqref{2.3} 
can be done in completely symmetric way, 
we shall only prove \eqref{2.2} in the following. 

\subsection{Construction of principal sets}\label{ssec2.1}
We now introduce the construction of principal sets as follows. 
Suppose that $P\in\cF_i$ satisfy 
$\mu(P)>0$ and, for some $k\in\Z$, 
$$
1_{P}2^{k-1}<1_{P}\cE_if\le 1_{P}2^k.
$$
We will write 
$\kp_1(P):=i$ and $\kp_2(P):=k$.
We define a stopping time 
$$
\tau_{P}
:=
1_{P}
\inf\{j\ge \kp_1(P):\,
\cE_jf>2^{\kp_2(P)+1}\}.
$$
We call a set $Q\subset P$ a principal set with respect to $P$ 
if it satisfies $\mu(Q)>0$ and 
there exist 
$j>\kp_1(P)$ and $l>\kp_2(P)+1$ 
such that 
$$
Q
=
\{2^{l-1}<1_{\{\tau_{P}=j\}}\cE_jf\le 2^l\}.
$$
Noticing that such $j$ and $l$ are unique, 
we will write 
$\kp_1(Q):=j$ and $\kp_2(Q):=l$.
Let $\cP^{*}(P)$ be 
the collection of all principal sets $Q$ 
with respect to $P$ and let 
$\cP(P):=\bigcup_{Q\in\cP^{*}(P)}Q$.
Then it is easy to see that they satisfy the following properties:

\begin{description}
\item[{\rm(iv)}] 
$\ds
1_{\{\kp_1(P)\le j<\tau_{P}\}}
\cE_jf
\le 2^{\kp_2(P)+1};
$
\item[{\rm(v)}] 
$\ds
\mu(\cP(P))\le2^{-1}\mu(P).
$
\end{description}

\noindent
Indeed, 
(v) follows from the use of weak-$(1,1)$ boundedness of Doob's maximal operator:
$$
\mu(\cP(P))
\le
2^{-\kp_2(P)-1}
\int_{P}f\,d\mu
=
2^{-\kp_2(P)-1}
\int_{P}\cE_if\,d\mu
\le 2^{-1}\mu(P).
$$

To construct the collection $\cP$ of principal sets, 
consider all $P\in\cF_{i_0}^0$ such that 
$\mu(P)>0$ and, for some $k\in\Z$, 
$$
P
=
\{2^{k-1}<\cE_{i_0}f\le 2^k\};
$$
that will be the first generation $\cP_1^{*}$ of principal sets. 
To obtain the second generation of principal sets, 
for each $P\in\cP_1^{*}$ 
we construct the collection
$\cP^{*}(P)$ of principal sets with respect to $P$, 
and define the second generation 
$$
\cP_2^{*}
:=
\bigcup_{P\in\cP_1^{*}}\cP^{*}(P).
$$
The next generations are defined inductively: 
$$
\cP_{n+1}^{*}
:=
\bigcup_{P\in\cP_n^{*}}\cP^{*}(P).
$$
We define the collection of principal sets $\cP$ by 
$$
\cP
:=
\bigcup_{n=1}^{\infty}\cP_n^{*}.
$$
We need the following lemma, 
which is a Carleson embedding theorem 
associated with the collection of principal sets $\cP$. 

\begin{lemma}\label{lem2.2}
We have 
$$
\sum_{P\in\cP}
\mu(P)2^{p(\kp_2(P)-1)}
\le 2
\|f\|_{L^p(d\mu)}^p.
$$
\end{lemma}

\begin{proof}
For $\lm>0$ let 
$$
\Gm_{\lm}
:=
\{P\in\cP:\,2^{\kp_2(P)-1}>\lm\}.
$$
Considering the maximal sets with respect to inclusion and 
using the property (v), we can write 
$$
\sum_{P\in\Gm_{\lm}}\mu(P)
\le 2
\sum_k\mu(P_k),
$$
where the sets $\{P_k\}\subset\Gm_{\lm}$ 
are nonoverlapping. Since 
$$
1_{P_k}\lm<1_{P_k}2^{\kp_2(P_k)-1}
<
1_{P_k}\cE_{\kp_1(P_k)}f
\text{ for all }P_k,
$$
we have 
$$
\sum_kP_k\subset\{f^{*}>\lm\},
$$
where $f^{*}$ is Doob's maximal operator defined by 
$$
f^{*}:=\sup_{i\in\Z}|\cE_if|.
$$
These imply 
$$
\sum_{P\in\Gm_{\lm}}\mu(P)
\le 2
\mu(\{f^{*}>\lm\})
$$
and hence 
$$
\sum_{P\in\cP}
\mu(P)2^{p(\kp_2(P)-1)}
\le 2
\|f^{*}\|_{L^p(d\mu)}^p.
$$
We have then the desired inequality by Doob's maximal theorem. 
\end{proof}

\subsection{Proof of \eqref{2.2}}\label{ssec2.2}
Using the principal sets $\cP$, 
we can decompose the left-hand side of \eqref{2.2} 
as follows: 
$$
\sum_{i\ge i_0}
\int_{\Omega}
1_{F_i}\al_i(\cE_if)(\cE_ig)\,d\mu
=
\sum_{P\in\cP}
\sum_{i\ge\kp_1(P)}
\int_{P\cap\{i<\tau_{P}\}}
1_{F_i}\al_i(\cE_if)(\cE_ig)\,d\mu.
$$
It follows from 
the property (iv), 
the condition \eqref{1.4} and 
H\"{o}lder's inequality that 
\begin{alignat*}{2}
\lefteqn{
\sum_{i\ge\kp_1(P)}
\int_{(P\setminus\cP(P))\cap\{i<\tau_{P}\}}
1_{F_i}\al_i(\cE_if)(\cE_ig)\,d\mu.
}\\ &\le 
2^{\kp_2(P)+1}
\int_{P}
\l(\sum_{i\ge\kp_1(P)}\al_i\r)
\l(\sup_{i\ge\kp_1(P)}1_{F_i}(\cE_ig)\r)
1_{P\setminus\cP(P)}
\,d\mu
\\ &\le
2^{\kp_2(P)+1}
\l\{\int_{P}
\l(\sum_{i\ge\kp_1(P)}\al_i\r)^q
\,d\mu\r\}^{\frac1q}
\l\{\int_{P\setminus\cP(P)}
\l(\sup_{i\ge\kp_1(P)}1_{F_i}(\cE_ig)\r)^{q'}
\,d\mu\r\}^{\frac1{q'}}
\\ &\le C
2^{\kp_2(P)+1}\mu(P)^{\frac1p}
\l\{\int_{P\setminus\cP(P)}
\l(\sup_{i\ge\kp_1(P)}1_{F_i}(\cE_ig)\r)^{q'}
\,d\mu\r\}^{\frac1{q'}}.
\end{alignat*}
The definition of $F_i$ enables us that
$$
\l(\sup_{i\ge\kp_1(P)}1_{F_i}(\cE_ig)\r)^{q'}
\le
\l(\sup_{i\ge\kp_1(P)}1_{F_i}(\cE_if)\r)^p.
$$
H\"{o}lder's inequality gives 
\begin{alignat*}{2}
\lefteqn{
\sum_{P\in\cP}
\sum_{i\ge\kp_1(P)}
\int_{(P\setminus\cP(P))\cap\{i<\tau_{P}\}}
1_{F_i}\al_i(\cE_if)(\cE_ig)\,d\mu
}\\ &\le C
\l\{\sum_{P\in\cP}\mu(P)2^{p(\kp_2(P)-1)}\r\}^{\frac1p}
\l\{\sum_{P\in\cP}
\l(\int_{P\setminus\cP(P)}(f^{*})^p\,d\mu\r)^{\frac{p'}{q'}}
\r\}^{\frac1{p'}}
\\ &\le C
\l\{\sum_{P\in\cP}\mu(P)2^{p(\kp_2(P)-1)}\r\}^{\frac1p}
\l\{\sum_{P\in\cP}
\int_{P\setminus\cP(P)}(f^{*})^p\,d\mu
\r\}^{\frac1{q'}}
\\ &\le C
\|f\|_{L^p(d\mu)}^{p\theta},
\end{alignat*}
where we have used 
Lemma \ref{lem2.2},
$\ds\frac{p'}{q'}\ge 1$,
the fact that the sets $P\setminus\cP(P)$ 
are nonoverlapping and Doob's maximal theorem. 

It follows also that 
\begin{alignat*}{2}
\lefteqn{
\sum_{i\ge\kp_1(P)}
\int_{\cP(P)\cap\{i<\tau_{P}\}}
1_{F_i}\al_i(\cE_if)(\cE_ig)\,d\mu.
}\\ &\le 
2^{\kp_2(P)+1}
\int_{P}
\l(\sum_{i\ge\kp_1(P)}\al_i\r)
\l(\sup_{\kp_1(P)\le i<\tau_{P}}1_{F_i}(\cE_ig)\r)
1_{\cP(P)}
\,d\mu
\\ &\le
2^{\kp_2(P)+1}
\l\{\int_{P}
\l(\sum_{i\ge\kp_1(P)}\al_i\r)^q
\,d\mu\r\}^{\frac1q}
\l\{\int_{\cP(P)}
\l(\sup_{\kp_1(P)\le i<\tau_{P}}1_{F_i}(\cE_ig)\r)^{q'}
\,d\mu\r\}^{\frac1{q'}}
\\ &\le C
2^{\kp_2(P)+1}\mu(P)^{\frac1p}
\l\{\int_{\cP(P)}
\l(\sup_{\kp_1(P)\le i<\tau_{P}}1_{F_i}(\cE_ig)\r)^{q'}
\,d\mu\r\}^{\frac1{q'}}.
\end{alignat*}
By the definition of $F_i$ 
$$
\l(\sup_{\kp_1(P)\le i<\tau_{P}}1_{F_i}(\cE_ig)\r)^{q'}
\le
\l(\sup_{\kp_1(P)\le i<\tau_{P}}1_{F_i}(\cE_if)\r)^p.
$$
We notice that 
$$
\l(\sup_{\kp_1(P)\le i<\tau_{P}}1_{F_i}(\cE_if)\r)^p
\le 2^{p(\kp_2(P)+1)}.
$$
These yield
$$
\l\{\int_{\cP(P)}
\l(\sup_{\kp_1(P)\le i<\tau_{P}}1_{F_i}(\cE_ig)\r)^{q'}
\,d\mu\r\}^{\frac1{q'}}
\le
\l\{\mu(P)2^{p(\kp_2(P)+1)}\r\}^{\frac1{q'}}.
$$
Thus, we obtain 
\begin{alignat*}{2}
\sum_{P\in\cP}
\sum_{i\ge\kp_1(P)}
\int_{\cP(P)\cap\{i<\tau_{P}\}}
1_{F_i}\al_i(\cE_if)(\cE_ig)\,d\mu
&\le C
\sum_{P\in\cP}
\l(\mu(P)2^{p(\kp_2(P)-1)}\r)^{\theta}
\\ &\le C
\l\{\sum_{P\in\cP}\mu(P)2^{p(\kp_2(P)-1)}\r\}^{\theta}
\\ &\le C
\|f\|_{L^p(d\mu)}^{p\theta},
\end{alignat*}
where we have used $\theta\ge 1$ and Lemma \ref{lem2.2}. 
The proof is now complete.

\end{document}